\documentclass[12pt]{amsart}

\usepackage{amsmath,amssymb,amsthm}
\usepackage{mathrsfs, wrapfig,graphicx, caption, cite, pifont}
\usepackage{paralist}
\usepackage[all]{xy}
\usepackage{enumerate}
\usepackage{hyperref}
\usepackage{color}

\numberwithin{equation}{section}

\newtheorem{theorem}{Theorem}[section]
\newtheorem{cor}[theorem]{Corollary}
\newtheorem{lemma}[theorem]{Lemma}

\newtheorem{prop}[theorem]{Proposition} \theoremstyle{definition}
\newtheorem{definition}[theorem]{Definition}
 \theoremstyle{remark}
\newtheorem{rem}[theorem]{Remark}

\newtheorem{exam}[theorem]{Example}
\newtheorem{lem}[theorem]{Lemma}

\DeclareMathOperator{\ran}{im} \DeclareMathOperator{\sign}{sign}
\DeclareMathOperator{\cd}{cd}
\DeclareMathOperator{\cfd}{cfd}

\newcommand{\be}{\begin{equation}}
\newcommand{\ee}{\end{equation}}
\newcommand{\dd}[2]{\genfrac{}{}{0pt}{}{#1}{#2}}
\newcommand{\ep}{\epsilon}

\def\beq{\begin{equation} } \def\eeq{\end{equation}}

 \def\qed{$\blacksquare$}

    \def\cE{ {\mathcal E} }

\def\cM{{\mathcal M}} \def\cN{ {\mathcal N} }

\def\cW{{\mathcal W}}
\def\cN{{\mathcal N}}

\def\tA{\tilde A}  
\def\M{{\rm red}}

\def\RR{\mathbb R}

\def\eps{\varepsilon}  

\def\ben{\begin{enumerate} }

\def\een{\end{enumerate} }

\begin{document}

\title{Determinant Expansions of Signed Matrices and of Certain
Jacobians}

\author{J. William Helton}
\address{Mathematics Department,
University of California at San Diego, La Jolla CA 92093-0112}
\email{helton@ucsd.edu}
\thanks{Partially supported by the NSF and the Ford Motor Co.}

\author{Igor Klep}
\address{Univerza v Ljubljani, IMFM,
Jadranska 19, SI-1111 Ljubljana, Slovenia}
\email{igor.klep@fmf.uni-lj.si}
\thanks{The author acknowledges the financial support from the state
budget by the Slovenian Research Agency (project No. Z1-9570-0101-06).}

\author{Raul Gomez}
\address{Mathematics Department,
University of California at San Diego, La Jolla CA 92093-0112}
\email{r1gomez@math.ucsd.edu}


\begin{abstract}
 This paper treats two topics: matrices with sign patterns
  and Jacobians of certain mappings. The main topic is
  counting the number of plus and minus coefficients in the determinant
  expansion of sign patterns and of these Jacobians.
The paper is motivated
 by an approach to chemical networks initiated by
Craciun and Feinberg.
We also give a graph-theoretic test for determining when
the Jacobian of a chemical reaction dynamics
has a sign pattern.
\end{abstract}

\maketitle

\section{Introduction}

\label{sec:Intro}

 This paper treats two topics: matrices with sign patterns
  and  Jacobians of certain mappings. The main topic is
  counting the number of plus and minus coefficients in their determinant
  expansion,  but other types of results occur along the way.
 It  is motivated
 by an approach to chemical networks initiated by
Craciun and Feinberg, see
\cite{CF05, CF06}, and extensions observed in \cite{CHWprept}.

\subsection{Determinants of Sign Patterns}
\label{subsec:Det}

The first topic, see \S \ref{sec:Nonsquare},
is purely matrix theoretic and generalizes
the classical theory
of sign definite matrices \cite{BS95}.
This subject considers classes of  matrices having a fixed sign pattern
(two matrices are in a given class iff each of their entries
has the same sign (or is 0)), then one studies determinants.
Call a {\bf sign pattern} a matrix $A$ with entries which are $\pm
A_{ij}$ or $0$, where $A_{ij}$ are free variables.
To a matrix $B$ we can associate its sign pattern $A={\rm SP}(B)$
with $\pm A_{ij}$ or $0$ in the correct locations.
If $A$ is square, then the determinant of $A$
is a polynomial in variables $A_{ij}$, which we call the
{\bf determinant expansion} of $A$.
We call a square invertible
matrix {\bf sign-nonsingular
(SNS)} if every term in the determinant expansion of its sign
pattern has the same
sign. There is a complete and satisfying
theory of these which associates a digraph to a square sign pattern
and a test which determines precisely if the matrix is SNS,
see \cite{BS95}.

In this paper we analyze square sign patterns and give
a graph-theoretic test to count the number
of positive and negative
signs in their determinant expansions;
Theorem \ref{thm:count}.
We extend the result
to nonsquare matrices and call our test on a matrix the
{\bf det sign test}.

\subsection{Jacobians of reaction form differential equations}
\label{subsec:Jac}

The second topic, \S \ref{sec:Jac},
in this paper applies this to systems of
ordinary differential equations which act on the
nonnegative orthant $\RR_{\geq 0}^d$ in $\mathbb R^d$:
\beq
\label{eq:de}\frac{dx}{dt}=   f(x),
\eeq
where $f:\RR_{\geq 0}^d\to\RR^d$.
The differential equations we
address are of a special form found in chemical
reaction kinetics:
\beq\label{eq:chemde}
\frac{dx}{dt}=Sv(x),
\eeq
where $S$ is a real $d\times d'$ matrix and $v$ is a column
vector consisting of $d'$ real-valued functions.
We say that system \eqref{eq:de} has {\bf reaction form} provided
it is represented as in \eqref{eq:chemde} with $v(x)=(v_1,\ldots,v_{d'})$
and
\begin{equation}
\label{def:vdepends}
v_j \text{ depends exactly on variables } x_i \text{ for which }
S_{ij}<0.
\end{equation}
Call $S$ the {\bf stoichiometric matrix} and the entries of $v(x)$
the {\bf fluxes}. We always assume the fluxes are continuously
differentiable.

Our second main result,
Theorem \ref{prop:JacSigns}, describes which $S$ have the property
that the Jacobian matrix $f'(x)= Sv'(x)$ has a sign pattern,
meaning that each entry $f'_{ij}(x)$ has sign independent of
$x$ in the orthant.
The characterization is graph-theoretic, clean and elegant.
The question was motivated
by works of Sontag and collaborators
\cite{AnS03,AS06,AS07}.

Our third main result here, Theorem \ref{thm:1cycle},
when specialized to square invertible $S$
counts the number of plus and minus signs in
the determinant expansion of
the Jacobian $f'(x)=Sv'(x)$ of a reaction form $f(x)=Sv(x)$
in the terms of a bipartite graph associated to $S$ and the
det sign test.
We use this to obtain results on the determinant expansion
for general nonsquare $S$.

We present many examples which illustrate features of our results
and limitations on how far one can go beyond them.

\subsection{Chemistry}
\label{subsec:Chem}

The reaction form differential equations
subsume chemical reactions where no
chemical appears on both sides of a reaction,
e.g.~catalysts.
Furthermore, in many situations
all fluxes $v_j(x) $ are monotone nondecreasing
in each $x_i$ when the other variables are fixed, that is,
$v'(x)$ has all entries nonnegative.
This happens in classical mass action kinetics
or for Michaelis-Menten-Hill type fluxes.
See \cite{Pa06} for an exposition.

A key issue with reaction form equations is how many equilibria
do they have in the strict positive orthant $\RR_{> 0}^d$.
It was observed in \cite{CF05,CF06,CF06iee}
that in many simple chemical reactions the
determinant of $f'$ has constant sign on the positive orthant
and as a consequence of a strong version
of this, any equilibrium which exists  is unique.
Other approaches exploiting this
determinant  hypothesis (under weaker assumptions)
are in \cite{BDB07,CHWprept}.
Roughly speaking,
if the determinant of the Jacobian $f'$ does not change
sign on a compact region $\Omega$,
then degree theory applies and bears effectively on this issue;
the full orthant $\RR_{>0}^d$ can easily be approximated by expanding
$\Omega$'s.

The degree argument is very flexible
and probably extends to many situations.
Fragile, however,
are establishing constraints on the sign of the determinant.
A key tool is
the {\bf determinant expansion} of \eqref{eq:chemde}, namely,
the expression $\det(S V(x))$ as a polynomial in the functions
$V_{ij}(x)$, which are the entries of the matrix function
$V(x)=v'(x)$.
The main issue is the sign of the terms in the determinant
expansion, are they all the same or if not are there few
``anomalous'' signs.
In \cite{CHWprept}
it is observed that in each example of Craciun and Feinberg
the determinant expansion has very few anomalous signs.
When this happens, then it gave some methods one could use
to prove existence and uniqueness of equilibria.
For example, if the determinant expansion has one minus sign
and many plus signs, and if the monotonicity condition
$V_{ij}(x) \geq 0$ holds, then $\det(S V(x))$ is positive on large
regions (which in particular situations can be estimated).

Our main results, Theorem \ref{thm:1cycle} etc., on $\det(S V(x))$
were motivated by a desire to develop
tools for  counting anomalous signs.
While the paper is not aimed at chemical applications, many
of the examples of matrices $S$ we use to illustrate our work
are stoichiometric matrices for chemical reactions.

The paper \cite{BDB07} identified chemical reaction
determinant expansions initiated by \cite{CF05,CF06}
with classical matrix determinant expansion theory
and sign patterns.
This is described in the book \cite{BS95} and pursued
into new directions in a variety of recent papers
such as \cite{BJS98,CJ06,KOSD07}.
The bipartite graph conventions in this paper are a bit
different than conventional, but were chosen to
be reasonably consistent with \cite{CF06}.

The authors wish to thank Vitaly Katsnelson for
diligent reading and suggestions.

\section{Matrices with sign patterns}
\label{sec:Nonsquare}

This section gives the set-up and our main results
on sign patterns as described in the introduction.

Let $t(A)$, respectively $m_\pm(A)$, denote the number of terms,
respectively $\pm$ signs, 
in the determinant expansion of the
square sign pattern $A$.
Recall {\bf a sign definite (SD)} matrix $A$ is one
with either $m_-(A)=0$ or $m_+(A)=0$ or $\det(A)=0$.
The number
of {\bf anomalous signs} $m(A)$ of a square sign pattern $A$ is
defined to be
$$
m(A):=\min\{m_-(A),m_+(A)\}.
$$
We say $A$ is {\bf $j$-sign definite} if it
has $j$ anomalous signs, that is $m(A)=j$.

\begin{lemma} \label{it:mpmn}\label{it:flipsign}
$m(A)=m(\tA)$ and
$m_+(A)=
m_-(\tA)$ if $\tA$ is obtained from $A$ by:\\ interchange of two
rows, or interchange of two columns, or multiplying a row by a
minus sign, or multiplying a column by a minus sign.  \end{lemma}

\begin{proof} Obvious. \end{proof}

{\bf Question (J-sign)}:
{\it Given a sign pattern $S$ we are interested in whether every
square submatrix is sign definite or more generally in getting an upper bound
$J$ on the $j$ for which $S$ contains a $j$-sign definite square matrix.}

We shall settle this question and give an even more refined
result for square sign patterns $A$ which counts $m_{\pm}(A)$.

\subsection{Basics on graphs, matrices and determinants}
\label{subsec:Graphs}

To this end we revert to graphs. Given a sign pattern $S$
let
$G(S)$ denote its {\bf signed bipartite graph}.
This is a simplified version of the species-reaction (SR) graph in
\cite{CF06}.
It is a signed bipartite graph with one set of vertices $C(S)$
based on columns and the other set of vertices $R(S)$ based on
rows. There is an edge joining column $c$ and row $r$ iff the
$(r,c)$ entry $S_{rc}$ of $S$ is nonzero.  The sign of this edge
is the sign of $S_{rc}$. If two edges meeting at the same column
have the same sign, they are called a {\bf c-pair}. By a
{\bf cycle} we mean
a closed (simple) path, with no other repeated vertices than the starting and ending vertices (sometimes also called a simple cycle, circuit, circle, or polygon).
A cycle that
contains an even (respectively odd) number of c-pairs is called an
{\bf e-cycle} (respectively {\bf o-cycle}). Recall a {\bf
matching} in a bipartite graph is a set of edges without common
vertices. Equivalently it is an injective mapping from one of the
vertex sets to the other.  A matching is called {\bf perfect} if
it covers all vertices in the smaller of the two vertex sets.  A
$k\times k$ square submatrix $A$ of $S$ corresponds to $k$ column
nodes $C(A)$ and $k$ row nodes $R(A)$; there is an associated
sub-bipartite graph $G(A)$ of $G(S)$.

\begin{exam} The following is an example taken from
\cite[Table 1.1.(v)]{CF05} which illustrates these definitions.
Given
$$
S=\left[
\begin{array}{rrrr}
 -1 & -1 & 0 & 0 \\
 -1 & 0 & 1 & 0 \\
 0 & -1 & -1 & -1 \\
 0 & 0 & -1 & 1 \\
 0 & 0 & 0 & -1 \\
 1 & 0 & 0 & 0 \\
 0 & 1 & 0 & 0
\end{array}
\right],
$$
the signed bipartite graph $G(S)$ is
as follows:
  
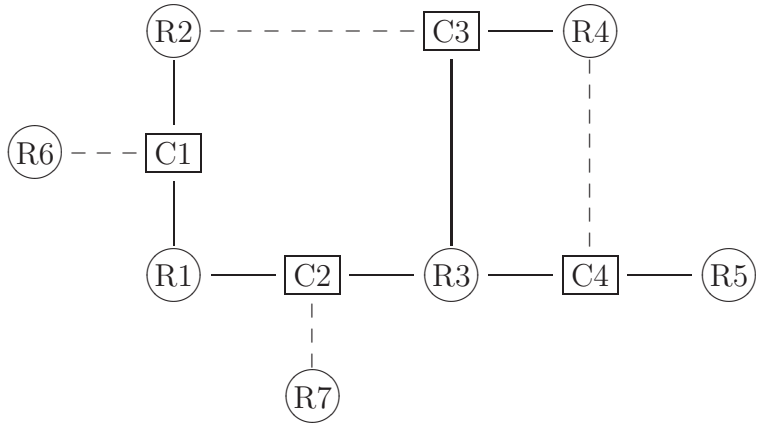
\[ \xymatrix{
& \begin{xy} *+{ \txt{R2} }* \frm{o} \end{xy} \ar@{{}-{}}[d]
\ar@{{}--{}}[rr] & & \fbox{C3} \ar@{{}-{}}[dd] \ar@{{}-{}}[r] &
\begin{xy} *+{ \txt{R4} }* \frm{o} \end{xy} \ar@{{}--{}}[dd] \\
\begin{xy} *+{ \txt{R6} }* \frm{o} \end{xy} \ar@{{}--{}}[r] &
\fbox{C1} \ar@{{}-{}}[d] \\ & \begin{xy} *+{ \txt{R1} }*
\frm{o} \end{xy} \ar@{{}-{}}[r] & \fbox{C2} \ar@{{}-{}}[r]
\ar@{{}--{}}[d] & \begin{xy} *+{ \txt{R3} }* \frm{o} \end{xy}
\ar@{{}-{}}[r] & \fbox{C4} \ar@{{}-{}}[r] & \begin{xy} *+{
\txt{R5} }* \frm{o} \end{xy} \\ & & \begin{xy} *+{ \txt{R7} }*
\frm{o} \end{xy} } \] 

\noindent Here the dashed lines denote positive edges and full lines
represent negative edges.

The edges C3$-$R4 and C3$-$R3 are a c-pair, while C3$-$R4 and
C3$-$R2 are not a c-pair.  The cycle C3$-$R4$-$C4$-$R3$-$C3 has one
c-pair, so is an o-cycle. On the other hand, the cycle
C1$-$R1$-$C2$-$R3$-$C3$-$R2$-$C1 has two c-pairs and so is an
e-cycle.  \qed \end{exam}

We use repeatedly the basic fact of linear algebra that if $A$ is
an $n \times n$ square matrix, then \beq \label{eq:detdef} \det A
= \sum_{\sigma \in S_n} \text{sign} (\sigma)
\prod_{i=1}^{n}A_{i,\sigma(i)}, \eeq where $S_n$ is the group of
permutations on $\{1, 2, \cdots, n \}$ and $A_{i,j}$ denotes the
$(i,j)$ term of $A$).

\begin{lemma}\label{lem:crucial} The bipartite graph $G(A)$ of a
square sign pattern $A$ has no perfect matching iff $\det A=0$.
\end{lemma}

\begin{proof}  $\det A\ne 0$ iff the determinant expansion will have at
least one nonzero term, say $A_{1,\sigma(1)}\cdots
A_{n,\sigma(n)}$ for some $\sigma\in S_{n}$. So all the
$A_{i,\sigma(i)}$ are nonzero, hence row $1$ is connected to
column $\sigma(1)$, $\ldots$, row $n$ is connected to column
$\sigma(n)$ and this yields a perfect matching for $G(A)$.
\end{proof}

\begin{rem}
\label{rem:perfMatch}
 The same argument shows that: \ben \item If the
bipartite graph of a rectangular matrix does not have a perfect
matching, then the determinants of all of its maximal square
submatrices are $0$.  \item The number of terms in the determinant
expansion of a square sign pattern $A$ is the number of perfect
matchings of $G(A)$.  \qed\een \end{rem}

\begin{rem} Note the following:
\ben \item Without loss of
generality for (J-sign) we can remove the second, third etc.~or any
colinear column from $S$.  Thus there are no colinear columns in
$S$.  This is true because, (a) if $A$ contains linearly dependent
columns then $\det A$ is 0.  (b) if $\tA$ is the same as $A$
except one column is removed and replaced (in any order) by a
scalar multiple of that column then the only possible change in
$m_\pm$ is $m_\pm(A) =  m_{\mp}(\tA)$.  \item Any cycle in $G(S)$
can be embedded in a square submatrix of $S$.  \qed\een \end{rem}

\subsection{SNS matrices vs.~e-cycles}
\label{subsec:ecycles}

Let us call a square invertible matrix {\bf sign-nonsingular
(SNS)} if every term in the determinant expansion of its sign
pattern has the same
sign \cite[Lemma 1.2.4]{BS95}. If all square submatrices of (a not
necessarily square matrix) S are either SNS or singular, then $S$
is {\bf strongly sign-determined (SSD)}.

\begin{prop}\label{prop:ssd}
A sign pattern $S$ is SSD iff the
signed
bipartite graph $G(S)$ has no e-cycle.  \end{prop}

\begin{proof}  This fact is essentially classical, cf.~\cite[Theorem
3.2.1]{BS95}.
Also, it is a special case of 
Theorem \ref{thm:count}.
\end{proof}

\subsection{Many Cycles: Square Matrices}
\label{subsec:graphicalTestSquare}

Next we turn to the more general situation where e-cycles occur in
$G(A)$.  The bipartite graph $G(A)$ enables us to count the number
of positive, negative and
anomalous signs in the determinant expansion of a square
sign pattern $A$.  We permute and re-sign to make sure all diagonal
entries of $A$ are negative.  Thus these diagonal entries
correspond precisely to a perfect matching $\cW$ in $G(A)$.  A
cycle in $G(A)$ that contains each edge $(c, \cW(c))$ in $G(A)$
corresponding to any column $c$ it touches, is called {\bf
interlacing} with respect to $\cW$, or {\bf $\cW$-interlacing} for
short.

\begin{rem}
To a given signed bipartite graph $G$ we can associate
(uniquely up to transposition and a permutation of rows
and columns) a
sign pattern $A$ with $G(A)=G$. The {\bf number of
anomalous signs} of a signed bipartite
graph $G$ with equipollent vertex sets is defined to be
$m(G)=m(A)$.
\end{rem}

\begin{exam}\label{ex:interlace}
Consider the following two graphs.  
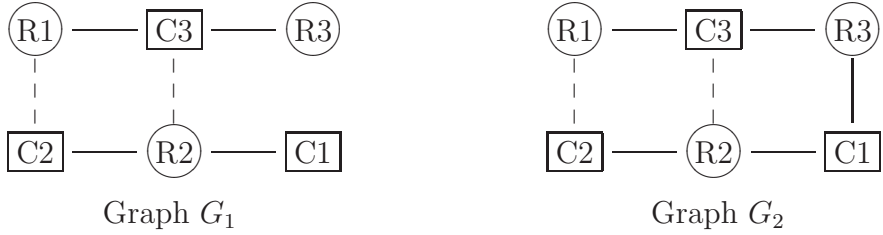
\[ \xymatrix{
\begin{xy} *+{\txt{R1} }* \frm{o}
\end{xy}  \ar@{{}--{}}[d] \ar@{{}-{}}[r] & \fbox{C3}
\ar@{{}--{}}[d]
&
\begin{xy} *+{ \txt{R3} }* \frm{o}
\end{xy} \ar@{{}-{}}[l]
\\ \fbox{C2} \ar@{{}-{}}[r] & \begin{xy}
*+{\txt{R2}}* \frm{o} \end{xy}
&
\fbox{C1}
\ar@{{}-{}}[l]
} \qquad \qquad\qquad
\xymatrix{
\begin{xy} *+{\txt{R1} }* \frm{o}
\end{xy}  \ar@{{}--{}}[d] \ar@{{}-{}}[r] & \fbox{C3}
\ar@{{}--{}}[d]
&
\begin{xy} *+{ \txt{R3} }* \frm{o}
\end{xy} \ar@{{}-{}}[l]
\\ \fbox{C2} \ar@{{}-{}}[r] & \begin{xy}
*+{\txt{R2}}* \frm{o} \end{xy}
&
\fbox{C1}
\ar@{{}-{}}[l] \ar@{{}-{}}[u]
}
 \]
\centerline{\hfill Graph $G_1$ \hfill\qquad\qquad Graph $G_2$\hfill}

\noindent Graph $G_1$ admits only one perfect matching $\cW$. Namely the set of edges
$\{$C1$-$R2, C2$-$R1, C3$-$R3$\}$. Hence its only cycle
R1$-$C3$-$R2$-$C2$-$R1 is not $\cW$-interlacing.
The sign pattern associated to $G_1$ is
$$
B=\left[\begin{array}{crr}
 0&B_{12}&-B_{13}\\
-B_{21}&-B_{22}&B_{23}\\
 0&\hfill 0\hfill&-B_{33}
\end{array}\right].
$$
As $\det(B)=-B_{12}B_{21}B_{33}$, $m(G_1)=m(B)=0$.

Graph $G_2$ on the other hand admits three perfect matchings.  For
instance, with respect to the matching $\{$C1$-$R3, C2$-$R2,
C3$-$R1$\}$, the cycle R1$-$C3$-$R2$-$C2$-$R1 is $\cW$-interlacing,
while the cycle C1$-$R3$-$C3$-$R2$-$C2$-$R1$-$C1 is not.
The sign pattern associated to $G_2$ is
$$
C=\left[\begin{array}{crr}
0&C_{12}&-C_{13}\\
-C_{21}&-C_{22}&C_{23}\\
-C_{31}&\hfill 0\hfill&-C_{33}
\end{array}\right].
$$
Since $\det(C)=-C_{12}C_{23}C_{31}-C_{12}C_{21}C_{33}+C_{13}C_{22}C_{33}$, $m(G_2)=m(C)=1$.

Note that the number of interlacing cycles depends on the
matching chosen. For instance, the graph $G_3$
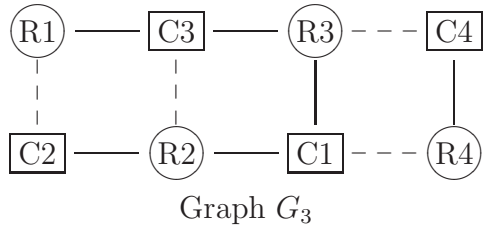
\[
\xymatrix{
\begin{xy} *+{\txt{R1} }* \frm{o}
\end{xy}  \ar@{{}--{}}[d] \ar@{{}-{}}[r] & \fbox{C3}
\ar@{{}--{}}[d]
&
\begin{xy} *+{ \txt{R3} }* \frm{o}
\end{xy} \ar@{{}-{}}[l]\ar@{{}--{}}[r]
&
\fbox{C4} \ar@{{}-{}}[d]
\\ \fbox{C2} \ar@{{}-{}}[r] & \begin{xy}
*+{\txt{R2}}* \frm{o} \end{xy}
&
\fbox{C1}\ar@{{}--{}}[r]
\ar@{{}-{}}[l] \ar@{{}-{}}[u]
& \begin{xy}
*+{\txt{R4}}* \frm{o} \end{xy}
} \]
\centerline{Graph $G_3$}

\noindent with the matching $\{$C1$-$R3, C2$-$R1, C3$-$R2, C4$-$R4$\}$
admits three interlacing cycles, while it has four cycles
interlacing 
with respect to the matching $\{$C1$-$R2, C2$-$R1, C3$-$R3, C4$-$R4$\}$.
\qed \end{exam}

The following theorem gives our {\it det sign test}
counting the number of signs in the
determinant expansion of a square sign pattern $A$ in
terms of $G(A)$.
For the sake of simplicity it is stated for matrices with
nonzero diagonal entries. This causes no loss of generality
since such a matrix can be obtained from any square invertible matrix
with a permutation of rows.

\begin{theorem}\label{thm:count} Let $A$ be a square
sign pattern
with nonzero diagonal elements. The diagonal gives us a perfect matching
$\cW$ that is fixed.
\ben\item[\rm (1)]
The number of terms,
$t(A)$, in
the determinant expansion of $A$ is one plus the cardinality of
the set of all sets of disjoint $\cW$-interlacing cycles of $G(A)$.
  \item[\rm (2)]
Let $\ep$ be the sign of the product of the diagonal elements
of $A$. Then the number of terms of
sign $-\ep$ in the determinant expansion of $A$, $m_{-\ep}(A)$, equals
the cardinality of the set of all sets of disjoint
$\cW$-interlacing cycles that contain an odd
number of $\cW$-interlacing e-cycles.
 \een
\end{theorem}

\begin{rem} By disjoint cycles we mean cycles with no common
vertices. The empty set is \emph{not} counted as a set of cycles.
\qed
\end{rem}

\begin{rem}
As observed in Example \ref{ex:interlace}, the number of
interlacing cycles depends on the matching $\cW$ chosen. However,
the numbers $t(A)$, $m_{\pm}(A)$ and $m(A)$ obtained from Theorem
\ref{thm:count}
are (clearly) independent of $\cW$.
\qed
\end{rem}

The special case of Theorem \ref{thm:count} where
$m(A)=0$ is settled by \cite[Theorem 3.2.1]{BS95} which is due
to Bassett, Maybee and Quirk \cite{BMQ68}.

The count of the signs in
the determinant expansion is simple in extreme cases, as the
following corollary shows.

\begin{cor} \label{cor:count} Let $A$ be a square sign pattern
with nonzero diagonal entries.  The diagonal induces a perfect matching
$\cW$.
Let $\ep$ denote the sign of the product of the diagonal elements
of $A$.
  \ben \item[\rm (1)] Suppose that $G(A)$ has $t$ cycles interlacing
with respect to $\cW$ and each pair of cycles has a nonempty
intersection.  Then the number of terms in the determinant
expansion of $A$ is $1+t$ and $m_{-\ep}(A)$
is the number of $\cW$-interlacing e-cycles.
\item[\rm (2)] Suppose that there are $t\geq 1$ cycles
of $G(A)$ each of which is $\cW$-interlacing and all are pairwise
disjoint. Then the number of terms in the determinant expansion of
$A$ is $2^t$ and
the number of anomalous signs is either $0$ $($if all $\cW$-interlacing
cycles are o-cycles$)$ or $2^{t-1}$. In the former case,
$m_{\ep}=2^t$ and in the
latter case $m_-(A)=m_+(A)=2^{t-1}$.
\een
\end{cor}

\begin{proof}  For (1) note that every set of disjoint interlacing cycles
contains only one cycle. (2) By Theorem \ref{thm:count}.(1),
the number of terms in the
determinant expansion of $A$ is just the number of all subsets of
$\{1,\ldots,t\}$, i.e., $2^t$.

For the second part of the claim we will compute the number
$m_{-\ep}(A)$. Let $r$ be the number of
e-cycles among the $t$ interlacing cycles. Of course, if $r=0$,
then there will be no anomalous signs. So assume $r>0$. There
are
$t-r$ interlacing o-cycles.  Since the
cycles are pairwise disjoint, we have by Theorem \ref{thm:count}
that a set consisting of some of the $t$ interlacing cycles
contributes a term with sign $-\ep$ to the determinant expansion of $A$
iff it contains an odd number of the $r$ e-cycles. Thus to
find $m_{-\ep}(A)$
we multiply the number of
ways we can choose an odd number of e-cycles from the $r$ e-cycles
by the number of ways we can choose any number of o-cycles from
the $t-r$ o-cycles.
The number of ways we can choose an odd number of
e-cycles from the $r$ e-cycles is $$\sum_{k=1}^{\lfloor
\frac{r-1}2 \rfloor}\binom{r}{2k+1}.$$ To simplify this, notice
that $0=(-1+1)^r = \sum_{k=0}^{r} \binom{r}{k}(-1)^k$ implies
$$\sum_{k=1}^{\lfloor \frac{r-1}2 \rfloor} \binom{r}{2k+1} =
\frac{1}{2}\sum_{k=0}^{r}\binom{r}{k} = 2^{r-1}.$$ The number of
ways we can choose a subset of o-cycles from the $t-r$ o-cycles is
$2^{t-r}$. Thus, $m_{-\ep}(A)=
2^{r-1}\cdot 2^{t-r}=2^{t-1}$ and
hence $m(A)=m_{\pm}(A)=2^{t-1}$.
\end{proof}

\begin{exam}
\label{ex:1badSign}
We now show how to determine when the determinant expansion of a square sign pattern
$A$ has no or one anomalous sign. Let us assume that all diagonal entries
of $A$ are nonzero and thus induce a perfect matching $\cW$. By Theorem
\ref{thm:count}.(2), $m(A)=0$ iff $G(A)$ contains no $\cW$-interlacing
e-cycles.

We claim that
{\it $m(A)=1$ iff $G(A)$ contains exactly one $\cW$-interlacing e-cycle
and no $\cW$-interlacing cycles disjoint from it.} Clearly, $(\Leftarrow)$
follows from Theorem \ref{thm:count}.
For the converse, note that if $G(A)$ contains at least two
$\cW$-interlacing e-cycles, then $m(A)\geq 2$ by Theorem
\ref{thm:count}.(2). Similarly we exclude the possibility of only
one $\cW$-interlacing e-cycle with other $\cW$-interlacing cycles disjoint
from it.
\qed
\end{exam}

\subsubsection{\bf Proof of Theorem \ref{thm:count}}

As preparation for the proof of the theorem, we briefly recall some
well-known facts about $S_n$.
A $S_n$-cycle $s=(s_{1}\, \cdots\,
s_{m})$ is the permutation mapping
$$
s_1\mapsto s_{2}\mapsto\cdots\mapsto s_{m}\mapsto s_{1}
$$
and fixing $\{1,\ldots,n\}\setminus\{s_{1}, \cdots\,
s_{m}\}$ pointwise.
To avoid collision with
cycles in various graphs appearing in the paper, we call
these cycles $S_n$-cycles.

\begin{exam}
For instance, the $S_{4}$-cycle $\sigma=(1\,2\,4)$ is the
mapping
$$
\left(
\begin{array}{c}
\xymatrix{
1 \ar@{|->}[d]&2 \ar@{|->}[d]&3 \ar@{|->}[d]&4 \ar@{|->}[d]\\
2&4&3&1
}
\end{array}
\right),
$$
while the mapping
$$\left(
\begin{array}{c}
\xymatrix{
1 \ar@{|->}[d]&2 \ar@{|->}[d]&3 \ar@{|->}[d]&4 \ar@{|->}[d]&5 \ar@{|->}[d]&6 \ar@{|->}[d]&7 \ar@{|->}[d]\\
2&4&5&1&3&6&7
}
\end{array}\right)
$$
can be written as $(1\,2\,4)(3\, 5)$.
\qed
\end{exam}

Every permutation $\sigma\in S_n$ can be written uniquely (up to the ordering in
the product) as a product of disjoint $S_n$-cycles.
Conversely, every set
of disjoint $S_n$-cycles gives a permutation in $S_n$.

\begin{lemma}\label{lem:detPerm}
If $\sigma=\tau_{1}\cdots\tau_{m}$ is a factorization
of $\sigma\in S_{n}$ into disjoint $S_{n}$-cycles,
and $\tau_{i}=(\tau_{i1}\, \tau_{i2}\,\cdots\,
\tau_{it_{i}})$, then
$$
{\rm sign}(\sigma)
\prod_{i=1}^{n} A_{i,\sigma(i)} =
\prod_{j=1}^{m} (-1)^{t_{j}-1} \prod_{i=1}^{t_{j}}
A_{\tau_{j\, i}, \tau_{j\,i+1}}
\prod_{k\not\in \{\tau_{ij}\}} A_{k,k}
$$
$($with the convention $\tau_{j\, t_j+1}=\tau_{j 1})$.
\end{lemma}

\begin{proof}
If $\sigma$ contributes to the determinant, then
every $\tau_{i}$ induces a cycle
of $G(A)$. For instance, the cycle $G(\tau_{i})$
corresponding to $\tau_{i}$ is defined to be the subgraph
$$
R(\tau_{i1}) - C(\tau_{i2}) -
R(\tau_{i2}) - C(\tau_{i3}) -
\cdots -
R(\tau_{it_{i}}) - C(\tau_{i1}) -
R(\tau_{i1})
$$
of $G(A)$.
The other ingredient is ${\rm sign}(\sigma)={\rm sign}(\tau_{1})\cdots
{\rm sign}(\tau_{m})$ and ${\rm sign}(\tau_{i})=(-1)^{t_{i}-1}$.
\end{proof}

\begin{proof}[Proof of Theorem {\rm \ref{thm:count}}]
Let $A$ be $n\times n$.
Statement (1) follows from Remark \ref{rem:perfMatch}.
To see why (2) is is true, we invoke the determinant
expansion formula \eqref{eq:detdef}.
For convenience we assume that all diagonal entries of
$A$ are negative.
Then $\ep=(-1)^n$ and we count the number of terms
with sign $-\ep$.
Each term
$x={\rm sign}(\sigma) \prod_{i=1}^{n}A_{i, \sigma(i)}$
in the expansion gives us a permutation $\sigma\in S_{n}$.
Since
every permutation can be written uniquely as a product
of disjoint $S_n$-cycles, we obtain a set of disjoint cycles
$\tau_{1},\ldots,\tau_{\ell}$ with $\sigma=\tau_{1}
\cdots\tau_{\ell}$. Say $\tau_{i}=(\tau_{i1}\, \tau_{i2}\,\cdots\,
\tau_{it_{i}})$. By the previous lemma,
$$
x= \prod_{j=1}^{\ell} (-1)^{t_{j}-1} \prod_{i=1}^{t_{j}}
A_{\tau_{j\, i}, \tau_{j\,i+1}}
\prod_{k\not\in \{\tau_{ij}\}} A_{k,k}.
$$
Observe that the sign of a product of the form
$ \prod_{i=1}^{t_{j}} A_{\tau_{j i}, \tau_{j\, i+1}}$
equals
$$
(-1)^{\text{number of c-pairs in } G(\tau_{i})}=:{\rm sign}(G(\tau_{i})).
$$
Taking into account that all diagonal entries are
negative, the sign of $x$ then equals the sign of
$$
(-1)^{\sum_{i=1}^{\ell}(t_{i}-1)} (-1)^{n-\sum_{i=1}^{\ell}t_{i}}
\prod_{i=1}^{\ell} {\rm sign}(G(\tau_{i})).
$$
This simplifies further to
$$
(-1)^{n-\ell} \prod_{i=1}^{\ell} {\rm sign}(G(\tau_{i})).
$$
In order for the term $x$ to have sign $(-1)^{n-1}$,
$(-1)^\ell \prod_{i=1}^\ell \sign (G(\tau_i))$
must not be equal 1.
We will show this is the case iff the number of e-cycles among
$\tau_0, \dots, \tau_\ell$
is odd.

\ben
\item[Case (1):] Suppose $\ell$ is odd.
Then
\begin{align*}
&(-1)^\ell \prod_{i=1}^\ell \sign (G(\tau_i))=-1
 \iff \prod_{i=1}^\ell \sign (G(\tau_i))=1\\
& \iff (-1)^{\# \text{(o-cycles among }\tau_0, \dots, \tau_\ell)}
=1\\
& \iff \# \text{ (o-cycles among }\tau_0, \dots, \tau_\ell) \text{ is even}\\
& \iff \# \text{ (e-cycles among }\tau_0, \dots, \tau_\ell) \text{ is odd},
\end{align*}
since $\ell$ is odd and $\#$ (e-cycles) $ = \ell- \#$ (o-cycles).

\smallskip
\item[Case (2):] Suppose $\ell$ is even.  Then
\begin{align*}
&(-1)^\ell \prod_{i=1}^\ell \sign (G(\tau_i))=-1 \iff
 \prod_{i=1}^\ell \sign (G(\tau_i))=-1\\
 &\iff (-1)^{\# \text{(o-cycles among }\tau_0, \dots, \tau_\ell)}
=-1\\
& \iff \# \text{ (o-cycles among }\tau_0, \dots, \tau_\ell) \text{ is odd}\\
& \iff \# \text{ (e-cycles among }\tau_0, \dots, \tau_\ell) \text{ is odd},
\end{align*}
since $\ell$ is even and $\#$ (e-cycles) $ = \ell- \#$ (o-cycles).
\qedhere
\een
\end{proof}

\subsection{Many Cycles: Nonsquare Matrices}
\label{subsec:graphicalTestNonSquare}

The graph-theoretic test described in
\S \ref{subsec:graphicalTestSquare}
gives a det sign test settling Question (J-sign).
In this section we extend the det sign test to nonsquare
sign patterns $S$.

A cycle has the property that the number of rows it
touches is the same as the number of columns it touches.
A set of cycles is called {\bf balanced}
if the
number of rows they touch is the same as the number
of columns they touch.
Every balanced set of cycles picks out a
square submatrix
$A$ of $S$ and hence induces a sub-bipartite
graph $G(A)$ of $G(S)$. Such a submatrix and
the sub-bipartite graph are both said to
be {\bf balanced}.
Note each column and row of $A$ appears in at least
one cycle in $G(A)$.

\begin{prop}\label{prop:a1}
For every square invertible submatrix $B$ of a
sign pattern $S$ there
is a balanced square submatrix $A$ of $S$ with
$m(A)=m(B)$. In fact, $A$ can be chosen to be
a submatrix of $B$.
\end{prop}

\begin{proof}
Suppose $B$ is the smallest square submatrix of $S$
violating the conclusion of the lemma. After permuting
rows we assume $B$ has nonzero entries on
the diagonal.
Since $B$ is not balanced, either a
row or a column of $B$ does not appear in
any cycle in $G(B)$.
Without loss of generality we assume
this to be row $1$.

Since we assume that row $1$ does not appear in any cycle of $G(B)$,
for $\sigma\in S_n$ with $\sigma=\tau_1\cdots\tau_\ell$,
where $\tau_i$ are disjoint $S_n$-cycles, the corresponding
term in the determinant expansion $x=\sign(\sigma)\prod_{i=1}^n
B_{i,\sigma(i)}$ will be zero if $1$ appears in one of the $\tau_i$.
Hence the nonzero terms $x$ will correspond to permutations $\sigma$
with $\sigma(1)=1$.
In other words, $B_{1,1}$ will get picked from row one.
So by removing row and column one from $B$
we obtain a smaller matrix $B_{0}$ with $m(B_{0})=m(B)$.
By the minimality assumption on $B$, there is
a balanced square submatrix $A$ of $B_{0}$ with
$m(A)=m(B_{0})=m(B)$, a contradiction.
\end{proof}

By this proposition, \textit{the answer $J$ to Question $($J-sign$)$ equals
the maximal number of anomalous signs obtainable from
a balanced square submatrix of the sign pattern $S$}. So the \textit{algorithm} for
finding the desired upper bound $J$ is as follows.
Consider sets of balanced cycles in $G(S)$. Each of these
induces a square submatrix $A$ of $S$.
If $G(A)$ admits no perfect matching, we continue with
another set of balanced cycles. Otherwise we count the number
of anomalous signs in $\det A$ by the procedure
described in Theorem \ref{thm:count} of \S \ref{subsec:graphicalTestSquare}.
The highest possible count obtained is the desired sharp upper
bound $J$.

\section{Reaction form differential equations and the Jacobians}
\label{sec:Jac}

Now we turn to studying
systems of
reaction form (RF)
ordinary differential equations which act on the
nonnegative orthant $\RR_{\geq 0}^d$ in $\mathbb R^d$:
\beq
\label{eq:RF}\frac{dx}{dt}=   f(x) = Sv(x),
\eeq
where $f:\RR_{\geq 0}^d\to\RR^d$,
$S$ is a real $d\times d'$ matrix and $v$ is a column
vector consisting of $d'$ real-valued functions.

The differential equation \eqref{eq:RF}
has {\bf weak reaction form} (wRF) provided
$V(x):= v'(x)$ satisfies
$S_{ij} > 0 \Rightarrow V_{ji}(x)=0$.
If a differential equation has wRF, then it has reaction form
provided
$S_{ij}=0 \Rightarrow V_{ji}(x)= 0$
and
$S_{ij}<0\Rightarrow V_{ji}(x)\ne 0$.
The flux vector $v(x)$ is {\bf monotone nondecreasing}
(respectively, {\bf monotone increasing})
if $\frac {\partial v_j}{\partial x_i}(x)$ is either $0$
for all $x\in\RR_{\geq 0}^{d'}$ or nonnegative (respectively,
positive) for all $x\in\RR_{> 0}^{d'}$.

This section analyzes two properties the Jacobian
of $f(x)$ might have. First we say exactly when
$f'(x)$ has a sign pattern (Theorem \ref{prop:JacSigns} and
Corollary \ref{cor:JacSigns}).
Secondly we give a method
based on \S \ref{sec:Nonsquare} for counting
the number of plus and minus coefficients in its
determinant expansion.

\subsection{Sign pattern of the Jacobian}
\label{subsec:SignPatJac}
\def\tS{\tilde S}

We first say precisely when the Jacobian of a
reaction form dynamics respects a sign pattern and find that
it does surprisingly often.

\begin{cor}\label{cor:JacSigns}
Given a reaction form differential equation
$$ \frac{dx}{dt} = S v(x)$$
with monotone increasing flux vector $v(x)$.
The Jacobian $Sv'(x)$
respects the same sign pattern for all $x \in \RR^n_{>0}$ if
the bipartite graph $G(S)$ 
does not contain a cycle of length four with
three negative edges.
Conversely, if $G(S)$ does contain such a cycle,
then some matrix $\tS$  arbitrarily close to $S$,
possibly $S$ itself,
 produces  $\tS v'(x)$ which fails to respect the same sign pattern
 for all $x$ in the orthant.
\end{cor}
The corollary is an immediate consequence of Theorem
\ref{prop:JacSigns} which operates at a higher level
of generality and requires the  definition
we now introduce.

Here and in the sequel, $U$ will denote the {\bf flux pattern}
assigned to $S$. It is a $d'\times d$ matrix with each entry
being $0$ or a free variable $U_{ij}$; the $(i,j)$th entry
of $U$ is 0 iff $S_{ji}\geq 0$.
 In case the differential equation
\eqref{eq:RF} satisfies RF and the flux vector $v(x)$
is monotone increasing, $U$ is the sign pattern of $V(x)$.

\begin{theorem}
\label{prop:JacSigns}
Let $S$ be a real $d\times d'$ matrix and $U$ the
corresponding flux pattern.
\ben
\item[\rm (1)]
The differential equation \eqref{eq:RF} has wRF iff
each diagonal term in $SU$ is a negative linear combination
of monomials in $U_{ij}$.
\item[\rm (2)]
$SU$ of a wRF differential equation
admits a sign pattern $($that is,
each entry of $SU$ is a positive or negative linear combination
of monomials in $U_{ij})$
whenever the matrix $S$ does not contain a $2\times 2$
submatrix with the same sign pattern as
\begin{equation}\label{eq:2cycle}
\left[\begin{array}{rr}+1&-1  \\-1,0&-1,0\end{array}\right]
\quad\text{or}\quad
\left[\begin{array}{rr}-1&+1  \\-1,0 &-1,0\end{array}\right]
\quad\text{or}\quad
\left[\begin{array}{rr}-1,0&-1,0\\+1&-1\end{array}\right]
\quad\text{or}\quad
\left[\begin{array}{rr}-1,0&-1,0 \\ -1&+1\end{array}\right].
\end{equation}
Here $-1,0$ stands for either $-1$ or $0$.
\item[\rm (3)]
SU of a RF differential equation admits a sign pattern iff
the matrix $S$ does not contain a $2\times 2$
submatrix with the same sign pattern as
\begin{equation}\label{eq:2cycleB}
\left[\begin{array}{rr}+1&-1\\-1&-1\end{array}\right]
\quad\text{or}\quad
\left[\begin{array}{rr}-1&+1\\-1&-1\end{array}\right]
\quad\text{or}\quad
\left[\begin{array}{rr}-1&-1\\+1&-1\end{array}\right]
\quad\text{or}\quad
\left[\begin{array}{rr}-1&-1\\-1&+1\end{array}\right].
\end{equation}
Equivalently, in terms of the bipartite graph,
$G(S)$ does not contain a cycle of length four with
three negative edges.
\item[\rm (4)]
The entry $(SU)_{ij}$ is nonzero iff there is some $k$ with
$S_{ik}\ne 0$ and $S_{jk}<0$.
If $SU$ admits a sign pattern, then
$\sign( (SU)_{ij})=\sign (S_{ik})$.
\een
\end{theorem}

\begin{proof}
(1)
Write $S=S_+-S_-$ for real matrices $S_+$, $S_-$ with nonnegative
coefficients satisfying the complimentarity property
$(S_+)_{ij} (S_-)_{ij}=0$. Diagonal entries of $SU$ are of the form
$\sum_{j} S_{ij}U_{ji}$ which meets the negative coefficient
condition iff $\sum_{j} (S_+)_{ij}U_{ji} = 0$ iff $(S_+)_{ij}U_{ji}=0$
for all $i,j$.
This uses that the $U_{ij}$ are free variables, so no
cancellation can occur.
Thus $(S)_{ij}>0$ iff $(S_+)_{ij} \neq 0$ implies $U_{ji}=0$
which is the wRF condition.

(2)
The $(i,j)$th entry of $SU$ does not have a
sign pattern iff $(S_+U)_{ij}\ne 0$ and $(S_-U)_{ij}\ne 0$.
$(S_+U)_{ij}=\sum_k (S_+)_{ik} U_{kj}$, so
$(S_+U)_{ij}\ne 0$ iff for some $k$, $(S)_{ik} > 0$ and
$U_{kj}\ne 0$,
i.e., $(S)_{ik} > 0$ and by wRF $S_{jk}\not >0$.
Similarly,
$(S_-U)_{ij}\ne 0$ iff there is some $\ell$ with
$(S_-)_{i\ell} \ne 0$ and  $U_{\ell j} \neq 0 $
so $S_{j\ell}\not >0$.
Taken together this implies that the $2\times 2$ submatrix
of $S$
given by rows $i,j$ and columns $k,\ell$ has the same sign
pattern as
one of the matrices in \eqref{eq:2cycle}.

(3)
This follows as in (2) by using that
if RF holds, then $U_{kj}\neq 0$ iff $S_{jk}<0$.
Also $U_{\ell j} \neq 0 $ iff $S_{j\ell}<0$.

(4) From $(SU)_{ij}=\sum_k S_{ik}U_{kj}$ it follows that
$(SU)_{ij}\ne 0$ iff there is $k$ with $S_{ik}\ne 0$ and
$U_{kj}\ne 0$. Due to the construction of $U$, $U_{kj}\ne 0$
iff $S_{jk}<0$. This proves the first part of the statement and
the second follows immediately since $U_{kj}$ is positive.
\end{proof}

Our next step is to introduce several
different types of determinant expansions.

\subsection{The core and other determinant expansions}
\label{subsec:coreDet}

For $S\in\RR^{d\times d'}$ with
$r:={\rm rank}(S)$
we define the {\bf core determinant} to be
\beq
\cd(S):=  \lim_{t\to 0} \frac{1}{t^{d-r}} \det ( SU  - t I).
\eeq

Let $B$ be a matrix whose range
is the orthogonal complement of the range of $S$.
We also use the formula
\beq
{\rm c}_0{\rm d}(S) :=\frac{ \det ( S U - B B^T)}{ \det (B B^T|_{(\ran S)^{\perp}})},
 \eeq
which does not depend on which $B$ we select and equals $\cd(S)$
(see Proposition
\ref{prop:BinetCauhr-ranks}), so we call both
the core determinant.
The {\bf Craciun-Feinberg determinant expansion} \cite{CF05}
is defined to be
$$
\cfd(S):= \det ( SU  - t I) \text{ with } t \text{ fixed, e.g.~}t= 1.
$$

For more details on the relationship between the Craciun-Feinberg
determinant expansion $\cfd(S)$ and the core determinant $\cd(S)$
we refer the reader to \S \ref{subsec:coreCF}.

\begin{prop}
\label{prop:BinetCauhr-ranks}
Let $C\in\RR^{d\times d'}$, let $D$ be a $d'\times d$ matrix with
possibly symbolic entries and suppose
$C$ has rank $r$. Define
$$
\alpha:=\frac{ \det ( CD - B B^T)}{ \det (B B^T|_{(\ran C)^{\perp}})},
$$
where $B$ is a matrix whose range is the orthogonal complement of
the range of $C$. Then
\ben
\item[\rm (1)]
$\alpha$ is independent of which matrix $B$ whose range
is the orthogonal complement of the range of $C$
is used to define it.
\item[\rm (2)]
$\alpha$ is the the determinant of the compression of $CD$
to the range of $C$.
\item[\rm (3)]
$$\alpha=\lim_{t\to 0} \frac{1}{t^{d-r}} \det ( CD  - t I).$$
\een
\end{prop}

\begin{proof}
We consider all matrices in the basis $\ran C \perp (\ran C)^{\perp}$. Then
$$CD=\left[\begin{array}{cc}CD|_{\ran C}&*\\0&0\end{array}\right] \text{ and }
BB^{T}=\left[\begin{array}{cc}0&*\\0&BB^{T}|_{(\ran C)^{\perp}}\end{array}\right].$$
Hence
$$CD-BB^{T}=\left[\begin{array}{cc}CD|_{\ran C}&*\\0&BB^{T}|_{(\ran C)^{\perp}}\end{array}\right].$$
This implies $\alpha=\det( CD|_{\ran S} )$ and is thus independent of $B$.
For (3), observe that
$$
CD  - t I =\left[\begin{array}{cc}(CD-tI)|_{\ran C}&*\\0&-t I|_{(\ran C)^{\perp}}\end{array}\right].$$
As the size of the second diagonal block is $(d-r)\times (d-r)$,
$$
\frac{1}{t^{d-r}} \det ( CD  - t I) = \det ((CD-tI)|_{\ran C}).
$$
Sending $t\to 0$ yields the desired conclusion.
\end{proof}

\begin{definition}
A column $c$ in $S\in\RR^{d\times d'}$ is called {\bf reversible} if
$-c$ is also a column of $S$.
(Many matrices coming from chemical reactions have
reversible columns.)
We call $-c$ the {\bf reverse} of $c$.
\end{definition}

\begin{exam}
\label{ex:CF3}
Let us consider an example which is a slight
modification of \cite[Table 1.1.(i)]{CF05}:
\nopagebreak\[
\xymatrix{
\begin{xy}
*+{ \txt{R5} }* \frm{o}
\end{xy}
\ar@{{}--{}}[r]
& \fbox{C3,C4} \ar@{{}-{}}[r]
&\begin{xy}
*+{\txt{R3} }* \frm{o}
\end{xy}  \\
&\begin{xy}
*+{\txt{R2}}* \frm{o}
\end{xy}
\ar@{{}-{}}[u]
& \ar@{{}--{}}[u] \fbox{C5}\\
\begin{xy}
*+{ \txt{R4} }* \frm{o}
\end{xy}  \ar@{{}--{}}[r]
& \ar@{{}-{}}[u] \fbox{C1,C2} \ar@{{}-{}}[r] &
\begin{xy}
*+{\txt{R1}}* \frm{o}
\end{xy}
\ar@{{}-{}}[u]
}
\]

\noindent The corresponding stoichiometric matrix $S$ and the vector
$v(x)$ are given by the following:
$$S=\left[
\begin{array}{rrrrrr}
 a_{11} & -a_{11} & 0 & 0 & -a_{13}\\
 a_{21} & -a_{21} & a_{22} & -a_{22} & 0  \\
 0 & 0 & a_{32} & -a_{32} & a_{33}  \\
 -a_{41} & a_{41} & 0 & 0 & 0  \\
 0 & 0 & -a_{52} & a_{52} & 0
\end{array}
\right], \qquad
v(x)=\left[
\begin{array}{l}
k_{1} x_{4}^{a_{41}} \\ k_{2} x_{1}^{a_{11}}
   x_{2}^{a_{21}}\\ k_{3} x_{5}^{a_{52}} \\ k_{4}
   x_{2}^{a_{22}} x_{3}^{a_{32}}\\k_{5}
   x_{1}^{a_{13}}
\end{array}
   \right].
$$
Note some of the columns of $S$ are reversible.
This phenomenon is captured in the graph
by listing two columns that are reverses of each other
in a common rectangular box.
For example, C3 and C4 appear in the same box
and in fact columns 3 and 4 are reverses of each
other. Sign of $S_{33}$ is the same as the sign
of $S_{23}$ and both appear in the graph as a solid line.
$S_{53}$ has sign opposite to these and so appears
in the graph as a dashed line.
This is also true for C4. Other dashed vs.~solid lines
of the graph coming from a box with reversible columns
follow the same pattern.

The corresponding $V(x)$
and $U$ are as follows:
$$
V(x)=\left[\begin{array}{ccccc}
 0 & 0 & 0 & x_{4}^{a_{41}-1} a_{41} k_{1} & 0 \\
 x_{1}^{a_{11}-1} x_{2}^{a_{21}} a_{11} k_{2} &
   x_{1}^{a_{11}} x_{2}^{a_{21}-1} a_{21} k_{2} & 0 & 0
   & 0 \\
 0 & 0 & 0 & 0 & x_{5}^{a_{52}-1} a_{52} k_{3} \\
 0 & x_{2}^{a_{22}-1} x_{3}^{a_{32}} a_{22} k_{4} &
   x_{2}^{a_{22}} x_{3}^{a_{32}-1} a_{32} k_{4} & 0 & 0
   \\
 x_{1}^{a_{13}-1} a_{13} k_{5} & 0 & 0 & 0 & 0
\end{array}\right],
$$
$$
U=\left[\begin{array}{ccccc}
 0 & 0 & 0 & U_{14}&0 \\
 U_{21}& U_{22}
 & 0 & 0
   & 0 \\
 0 & 0 & 0 & 0 & U_{35}\\
 0 & U_{42} &
   U_{43} & 0 & 0
   \\
 U_{51} & 0 & 0 & 0 & 0
\end{array}\right].
$$
For generic choices of the numbers $a_{ij}$ the matrix
 $S$
  will be of rank $3$ and this is
what
we
focus on.
A straightforward computation gives
\begin{eqnarray*}
\cd(S)&=&
-2 a_{13} a_{41} a_{52} U_{14} U_{35} U_{51}-2
   a_{13} a_{21} a_{52} U_{22} U_{35} U_{51}-2
   a_{13} a_{22} a_{41} U_{14} U_{42} U_{51}\nonumber\\
&& -2 a_{13} a_{32} a_{41} U_{14} U_{43} U_{51}
-2 a_{13} a_{21} a_{32} U_{22} U_{43} U_{51}+2
   a_{11} a_{22} a_{33} U_{22} U_{43} U_{51}\nonumber
\end{eqnarray*}
Hence there is potentially one anomalous sign in $\cd(S)$. However,
\begin{eqnarray*}
-2 a_{13} a_{21} a_{32} U_{22} U_{43} U_{51}+2
   a_{11} a_{22} a_{33} U_{22} U_{43} U_{51} &=&
2 (a_{11} a_{22} a_{33} -  a_{13} a_{21} a_{32}) U_{22} U_{43} U_{51}
\end{eqnarray*}
so $\cd(S)$ has one, respectively no anomalous sign, depending on
whether $a_{11} a_{22} a_{33} -  a_{13} a_{21} a_{32}$ is positive,
respectively nonpositive.
\qed
\end{exam}

\begin{exam}
Suppose $S=\left[\begin{array}{rr}
-a_{11}&a_{11}\\
-a_{21}&a_{21}
\end{array}\right]$. Then $SU-I$ admits a sign pattern; it is
a $2\times 2$ matrix with all entries negative. Hence the determinant
expansion of its sign pattern has one anomalous sign by the det
sign test. However,
$\cfd(S)$ and $\cd(S)$ have no anomalous signs.
\qed
\end{exam}

\subsection{Formulas for determinants of products of matrices}
\label{subsec:BinetCauchy}

For a matrix $A$, $A(\alpha|\delta)$ will refer to the submatrix
of $A$ with rows indexed by $\alpha$ and columns indexed by $\delta$.

\def\de{\delta}
\def\al{\alpha}
\def\all{\rm all}

Recall the
{\bf Binet-Cauchy formula} for the determinant of the product $AB$
of a $m \times n$ matrix $A$ and a $n\times m$ matrix $B$:
\beq
\label{eq:Binet-Cauchy}
\det( AB) = \sum_{\dd{\delta\subseteq\{1,\ldots,n\}}{|\de|=m}}
\det( A(\all|\de))\ \det (B(\de|\all)).
\eeq
(If $m > n$, then there is no admissible set $\de$ and the determinant $\det(AB)$ is zero.)

Combining Proposition \ref{prop:BinetCauhr-ranks}
with the Binet-Cauchy formula we obtain

\begin{lemma}
\label{lem:rankdet}
For $S\in\RR^{d\times d'}$ having rank $r$,
the core determinant is given by
\beq
\label{eq:rankdet}
 \cd(S)  = (-1)^{d-r}
\sum_{|\al|,|\beta|= r } \det (S (\al|\beta)) \  \det (U(\beta|\al))
.
\eeq
\end{lemma}

\begin{proof}
Use \eqref{eq:Binet-Cauchy} and
$
\left[
\begin{array}{cc}
 PQ &GH
\end{array}
\right]
=
\left[%
\begin{array}{cc}
  P & G \\
\end{array}%
\right]
\left[%
\begin{array}{c}
  Q \\
  H
\end{array}%
\right]
$
to get
\beq
\label{eq:keyDetExp}
 \det ( S U - t  I)  =
\sum_{|\de|=d } \det
\left(
\left[
\begin{array}{cc}
 S &-tI
\end{array}
\right]
(\all|\de)\right) \; \det \left( \left[\begin{array}{c}
U\\ I\end{array}\right](\de|\all)\right),
\eeq
where $d$ is the number of rows of $S$.
Since rank $S$ is $r$, $t^{d-r}$
factors out of $\det\left(\left[
\begin{array}{cc}
 S &-tI
\end{array}
\right](\all|\de)\right)
$, so
$
 \lim_{t\to 0} \frac{1}{t^{d-r}} \det \left(\left[
\begin{array}{cc}
 S &-tI
\end{array}
\right](\all|\de)\right)
$
exists. Let us look
at terms of degree $d-r$ in $t$ in \eqref{eq:keyDetExp}.
$
 \det \left(\left[
\begin{array}{cc}
 S &-tI
\end{array}
\right](\all|\de)\right)
$
will be of degree $d-r$ in $t$ iff $\delta$ will consists
of exactly $r$ columns $\beta$ of $S$. If $\al$ denotes
the set of rows of $S$ that do not hit any of the
columns of $-tI$ chosen by $\beta$, then
$$
 \det \left(\left[
\begin{array}{cc}
 S &-tI
\end{array}
\right](\all|\de)\right)
=
(-1)^{d-r} t^{d-r}
\det (S (\al|\beta)).
$$
It is clear that such pairs $(\al,\beta)$ are in a
bijective correspondence with all $\de$ that pick
$r$ columns of $S$. Hence
$$
\det(SU-tI)=  (-t)^{d-r} 
\sum_{|\al|,|\beta|= r }
\det (S (\al|\beta))\  \det (U(\beta|\al)) + (\text{higher order
terms in }t).
$$
Dividing by $t^{d-r}$ and sending $t\to 0$ proves \eqref{eq:rankdet}.
\end{proof}

Formulas (\ref{eq:rankdet}) and \eqref{eq:keyDetExp} are in contrast to $\cfd(S)$ which is given by the more complicated
expression
\beq
\label{eq:cfdet}
\cfd(S)  =
\sum_{s=1}^r \sum_{|\al|=|\beta|=s}
(-t)^{d-s} \det (S (\al|\beta)) \  \det (U(\beta|\al))
\eeq
The fact is known (cf.~\cite{CF05}, \cite[proof of Theorem 4.4]{BDB07})
and its proof follows the line
of the proof of Lemma \ref{lem:rankdet}.

For the chemical interpretation of the core determinant
vs.~the Craciun-Feinberg determinant see our \S \ref{subsec:coreCF}.

\begin{lemma}
\label{lem:reduceBadSigns}
The number of anomalous signs in $\cd(S)$ is at most the
number of anomalous signs in $\cfd(S)$.
\end{lemma}

\begin{proof}
By looking at the formulas \eqref{eq:rankdet} and
\eqref{eq:cfdet} it is clear that each term appearing
in $\cd(S)$ also appears (multiplied with $t^{d-r}$) in $\cfd(S)$.
Terms $w$ from $\cd(S)$ have degree $r$ in the $U_{ij}$'s.
All terms in $\cfd(S)$ not coming from terms in $\cd(S)$
have degree $<r$ in the $U_{ij}$'s. Thus there is no
cancellation and the
statement follows.
\end{proof}

\begin{rem}
Example \ref{ex:tail} shows that the number of anomalous signs
in $\cd(S)$ can be strictly smaller than the number
of anomalous signs in $\cfd(S)$.

If $B$ is the sign pattern associated to the graph $G_1$ of Example \ref{ex:interlace},
and $S=\left[\begin{array}{cc}B&-B\end{array}\right]$,
then $\cd(S)$ has no anomalous signs, whereas $\cfd(S)$ has
one anomalous sign. We leave this as an exercise
for the interested reader.
\qed
\end{rem}

\subsection{Generic matrices and the reduced
$S$-matrix}

In this section we introduce some basic definitions
and illustrate them with an example.

\begin{definition}
A matrix $A$ is called {\bf weakly generic} if its
rank $r$ is maximal among all matrices with the same
sign pattern.
If, in addition, all $r\times r$ submatrices of $A$ are
weakly generic, then $A$ is called {\bf generic}.

The set of all (weakly) generic $m\times m$ matrices with a given sign
pattern is open and dense in the set of all $m\times m$ matrices
with that sign pattern.
\end{definition}

\begin{lem}
\label{lem:genericRank}
The rank $r$ of a generic matrix $A$ with connected
graph $G(A)$ is equal to the minimum of the number of rows
or of columns of $A$. If $G(A)$ has $\ell$ components
$G_1,\ldots,G_\ell$
and $r_i$ is the minimal number of column or row nodes in
$G_i$, then $r=\sum_{i=1}^\ell r_i$.
\end{lem}

\begin{proof}
Obvious.
\end{proof}

\begin{definition}
For $S\in\RR^{d\times d'}$
let $S_\M$ denote a {\bf reduced $S$-matrix},
i.e., a matrix obtained from $S$ by removing one column
out of every pair of columns which are reverses of each other.
Clearly, $S_\M$ contains no reversible columns.
The {\bf reduced flux pattern} $U_\M $ is obtained from
$S$ and $S_\M$:
it is built from the sign pattern of $-S_\M^T$ by
setting all entries coming from positive entries
in columns nonreversible in $S$ to $0$.
In particular, if all columns of $S$ are reversible,
then $U_\M $ is the sign pattern of $-S_\M^T$.
If no column of $S$ is reversible, then $S_\M=S$ and
$U_\M =U$.
\end{definition}

\begin{exam}
\label{ex:CF3-again}
Let us revisit Example \ref{ex:CF3}.
A reduced $S$-matrix $S_\M$ and the reduced flux
pattern $U_\M $ are
$$S_\M=\left[
\begin{array}{rrrrrr}
 a_{11} &  0 & -a_{13}\\
 a_{21} &  a_{22} & 0  \\
 0 &  a_{32} &  a_{33}  \\
 -a_{41}  & 0 & 0  \\
 0 &  -a_{52} & 0
\end{array}
\right], \qquad
U_\M =\left[
\begin{array}{ccccc}
-U_{11}&-U_{12}&0&U_{14}&0\\
0&-U_{22}&-U_{23}&0&U_{25}\\
U_{31}&0&0&0&0
\end{array}\right].
$$
In most cases $S_\M$ will be generic and hence of rank $3$.
By a straightforward computation,
$$
S_\M U_\M  = \left[\begin{array}{ccccc}
-a_{11}U_{11}-a_{13}U_{31}&-a_{11}U_{12}&0&a_{11}U_{14}&0\\
-a_{21}U_{11}&-a_{21}U_{12}-a_{22}U_{22}&-a_{22}U_{23}&a_{21}U_{14}&a_{22} U_{25}\\
a_{33}U_{31}&-a_{32}U_{22}&-a_{32}U_{23}&0&a_{32} U_{25}\\
a_{41}U_{11}&a_{41}U_{21}&0&-a_{41}U_{14}&0\\
0&a_{52}U_{22}&a_{52}U_{23}&0&-a_{52} U_{25}
\end{array}\right]
$$
After a possible renaming of the free variables in $U_\M $,
$SU=S_\M U_\M $. This is the {\it key observation} we use in
the next sections in order to
count or estimate the number of anomalous signs in $\cd(S)$.
\qed
\end{exam}

\begin{lemma}\label{lem:key}
If $S$ is a real $d\times d'$ matrix,
if $S_\M$ is any reduced $S$-matrix
and $U_\M$ the corresponding reduced flux pattern,
then
$$
SU=S_\M U_\M $$
after a possible renaming of the free variables in $U_\M $.
\end{lemma}

\begin{proof}
Suppose first that $S=\left[\begin{array}{cc}S_\M&-S_\M\end{array}\right]$.
The corresponding matrix $U$ is of the form
\begin{equation}\label{eq:Ux}
U=\left[\begin{array}{c} U_0\\U_1 \end{array}\right].
\end{equation}
The sign pattern of $U_0^T$ is the same
as that of $-S_\M$. Furthermore, nonzero entries of $U_0^T$
coincide with negative entries of $S_\M$. Similarly,
nonzero entries of $U_1^T$ coincide with positive entries of
$S_\M$.

Clearly, $SU=S_\M(U_0-U_1)=S_\M U_\M $ (after a possible renaming of
the free variables in $U_\M $).

Let us now look at the general case, where some of the columns
do not have reverses in $S$.
We `expand' $S$ to
$
\tilde S=\left[\begin{array}{cc} S_\M & -S_\M\end{array}\right]
$
by adding reverses of nonreversible columns.
We insert rows of zeros at the appropriate places in $U$.
Again, we write
$$
\tilde U=\left[\begin{array}{c} U_0\\U_1 \end{array}\right].
$$
As before, nonzero entries of $U_0^T$ correspond to negative
entries of $S_\M$.
(Nonzero entries of $U_1^T$ correspond to a subset of the set of
all positive entries of $S_\M$.)
As $U_0-U_1=U_\M $, this concludes the proof.
\end{proof}

\subsection{Counting anomalous signs
when $S_\M$ is square}

Now we give our main theorem for square reduced $S$-matrices.
The result is strong and effectively reduces the problem
to the matrix and graph-theoretic test of \S
\ref{subsec:graphicalTestNonSquare}.

\begin{theorem}\label{thm:1cycle}
Let $S$ be a real $d\times d'$ matrix and
suppose $S_\M$ is a generic square invertible matrix.
Then:
\ben
\item[\rm (1)]
The number of terms in the core determinant $\cd(S)$ equals the
number of terms
in the determinant expansion of $U_\M $.
\item[\rm (2)]
The number of anomalous
signs of the core determinant $\cd(S)$ is the number
of anomalous
signs in the determinant expansion of $U_\M $.
\een
\end{theorem}

\begin{rem}\label{rem:pm-count}
Note that the theorem gives a count of positive
and negative terms in $\cd(S)$ when combined with Theorem
\ref{thm:count}.
The number of (anomalous) signs
in the determinant expansion of $U_\M $ is bounded
above by the number of (anomalous) signs
in the determinant expansion of the sign pattern of
$S_\M$.
\qed
\end{rem}

\begin{proof}[Proof of Theorem {\rm \ref{thm:1cycle}}]
Follows immediately from Lemma \ref{lem:key}.
\end{proof}

\subsection{Rectangular $S_\M$ matrices}

This section gives results and examples for the case of
rectangular reduced $S$-matrices. Our theorem for
complicated situations would not easily yield the precise count.
On the other hand, it yields estimates and in various simple
cases it is effective.

We can use Lemma \ref{lem:key} to
provide a Binet-Cauchy expansion with fewer terms than
there were in
Lemma \ref{lem:rankdet}, namely:

\begin{lemma}
\label{lem:rankdet2}
For $S\in\RR^{d\times d'}$ having rank $r$,
the core determinant is given by
\beq
\label{eq:rankdet2}
 \cd(S)  = (-1)^{d-r}
\sum_{|\al|,|\beta|= r } \det (S_\M (\al|\beta)) \  \det (U_\M (\beta|\al)).
\eeq
\end{lemma}

Theorem \ref{thm:1cycle} and Remark \ref{rem:pm-count}
tell us how to count
the number of positive, negative or
anomalous signs in $\cd(S)$ with generic $S_\M$. By the
Binet-Cauchy formula \eqref{eq:rankdet2} given in Lemma
\ref{lem:rankdet2} we count the number of positive and negative
terms for each of the
$\det (U_\M (\beta|\al))$ and take into account
the sign of $\det (S_\M (\al|\beta))$.
The sum of these will give us a count
for the number
of positive and negative terms in $\cd(S)$.
Note: due to the freeness of entries of $U$, there is
no cancellation between the summands.
In particular, this count gives us a lower
bound and upper bound on the number of anomalous signs in $\cd(S)$.\label{blah}

\begin{theorem}
\label{thm:lowBound}
Suppose $S\in\RR^{d\times d'}$ has rank $r$.
Let $S_\M$ be a reduced $S$-matrix and
$U_\M $ the reduced flux pattern.
Suppose that $S_\M$ is generic.
\ben
\item[\rm (1)]
The number of anomalous signs in $\cd(S)$
is at least
$$
\sum_{|\al|,|\beta|= r } m (U_\M  (\beta|\al))
$$
and at most
$$
\sum_{|\al|,|\beta|= r } t(U_\M  (\beta|\al))
 - m (U_\M  (\beta|\al))
.
$$
\item[\rm (2)] The number of terms of sign
$(-1)^{d-1}$ in $\cd(S)$ is
at least
$$
\sum_{S_\M(\al |\beta) \in\cN } m (U_\M  (\beta|\al))
$$
and at most
$$
\sum_{S_\M(\al |\beta)\in\cN } t(U_\M  (\beta|\al))
 - m (U_\M  (\beta|\al))
,
$$
where $\cN$ is the set of all $r\times r$ submatrices
$S_\M$ that are not SD.\een
\end{theorem}

\begin{proof}
(1) follows from the explanation given above, so we consider
(2). For a SD matrix $S_0=S_\M(\al|\beta)$, all terms in the
determinant expansion of the sign pattern of $S_0$ have
the same sign. Hence the same holds true for $U_0=U_\M(\beta|\al)$
which is the sign pattern of $-S_0^T$ with possibly some
entries set to $0$. If $|\al|=|\beta|=r$, then the sign
of $\det(U_0)$ is $0$ or $(-1)^r$ times the sign of $\det(S_0)$. Hence by Lemma
\ref{lem:rankdet2}, a term of sign $(-1)^{d-1}$
in $\cd(S)$ cannot come from a $r\times r$ SD submatrix of
$S_\M$.
To conclude the proof, note that given
$S_i=S_\M(\al|\beta) \in\cN$, the term
$\det (S_\M(\al|\beta))\, \det( U_\M(\beta|\al))$ will contribute
at least $\min\{m_-(U_i),m_+(U_i)\}=m(U_i)$ terms of
sign $(-1)^{d-1}$ in $\cd(S)$ and at most
$\max\{m_-(U_i),m_+(U_i)\}=t(U_i)-m(U_i)$ terms of
sign $(-1)^{d-1}$. (Here $U_i:= U_\M(\beta|\al)$.)
\end{proof}

These bounds are often tight as the next examples illustrate.

\subsubsection{\bf Examples}

\begin{exam}
\label{ex:recipe0}
Let $S$ be a $d\times d'$ matrix of rank $r$ with generic $S_\M$.\\
{\it Suppose $S_\M$ has no e-cycle interlacing with respect to a
perfect matching,
then $\cd(S)$ has no
anomalous signs.}\\
 This we now demonstrate.
 By assumption and Theorem \ref{thm:count},
any $r\times r$ submatrix
$S_0 = S_\M(\alpha|\beta)$ of $S_\M$ is SD. Then $\cN=\varnothing$,
so by Theorem \ref{thm:lowBound}, $\cd(S)$ will have no anomalous
signs.

Conversely, {\it if $\cd(S)$ has no anomalous signs,
then $G(U_\M )$ has no
e-cycles interlacing with respect to a perfect matching.}\\
To see why this is true, we invoke Theorem \ref{thm:count}.
Such an e-cycle and the perfect matching in $G(U_\M )$ pick out
a $r\times r$ submatrix $U_\M (\beta|\alpha)$ of $U_\M $. The corresponding
summand in \ref{eq:rankdet2} will then yield at least one
anomalous sign by Theorem \ref{thm:count}.

In the fully reversible case this yields a necessary and sufficient
condition for $\cd(S)$ to have no anomalous signs.
\qed
\end{exam}

\begin{rem}
We recall that for $\cfd(S)$ what we have
just done is known \cite{CF06} in the fully
reversible case $S=\left[\begin{array}{cc} S_\M & -S_\M\end{array}\right]$. What is shown in \cite{CF06}, implies that $\cfd(S)$ has no
anomalous signs iff $G(S_\M)$ has no e-cycles.
\qed
\end{rem}

\begin{exam}
\label{ex:tail}
Suppose 
$S_\M$ is generic,
$S=\left[\begin{array}{cc}S_\M&-S_\M\end{array}\right]$
and the graph $G(S_\M)$ is:
$$
\xymatrix{
\begin{xy} *+{\txt{R1} }* \frm{o}
\end{xy}  \ar@{{}--{}}[d] \ar@{{}-{}}[r] & \fbox{C1}
\ar@{{}--{}}[d]
\\ \fbox{C2} \ar@{{}-{}}[r] & \begin{xy}
*+{\txt{R2}}* \frm{o} \end{xy}\ar@{{}-{}}[r]
& \fbox{C3} \ar@{{}--{}}[r] & \begin{xy}
*+{\txt{R3}}* \frm{o} \end{xy}\ar@{{}-{}}[r]
& {\cdots}\ar@{{}-{}}[r]  &
 \fbox{C$n$} \ar@{{}--{}}[r] & \begin{xy}
*+{\txt{R$n$}}* \frm{o} \end{xy}
}
$$

\noindent
There are $n$ rows and $n$ columns, so
the rank of $S$ (and $S_\M$) is $n$.
$G(S_\M)$ supports exactly one rank $n$ square matrix, $S_\M$
itself.

$G(S_\M)$ has one cycle with no c-pairs,
so it is an e-cycle.
$G(S_\M)$ admits 2 perfect
matchings:  the e-cycle interlaces both matchings.
So by Theorem \ref{thm:count}, we get  $\cN$ is $S_\M$.

Theorem \ref{thm:lowBound} together with
 Theorem \ref{thm:count} imply
$$
1 = m(U_\M) = \sum_{ U_i\in\cN } m (U_i)
\leq 
m(S)   \leq 
\sum_{ U_i\in\cN }
[ t(U_i) - m(U_i) ]= t(U_\M) - m(U_\M) =2 - 1
 = 1.
$$
Thus generically $\cd(S)$ has one anomalous sign (independent of $n\geq 2$).

Alternative to Theorem \ref{thm:lowBound},
since $S_\M$ is square, we could have used
Theorems \ref{thm:1cycle} and \ref{thm:count} which
tell us  that $\cd(S)$ will have 2 terms,
one with a positive and one with a negative sign.

On the other hand, the number of anomalous signs in $\cfd(S)$
increases rapidly with $n$.
$$
\begin{array}{c|c}
& \# \text{ of anomalous} \\
n & \text{signs in }\cfd(S)
\\
\hline
2&1 \\
3&2 \\
4&5 \\
5&13 \\
6&34 \\
7&89 \\
8&233 \\
9&610\\
10&1597

\end{array}
$$
This data is consistent with
$$
\text{number of anomalous signs}= {\rm Fib}(2n-3)
$$
(see the website \url{http://www.research.att.com/~njas/sequences/}).
We leave it to the interested reader to see if this is
true.
\qed

\end{exam}

\begin{exam}
Suppose
$S=\left[\begin{array}{cc}S_\M&-S_\M\end{array}\right]$,
$S_\M$ is generic
and $G(S_\M)$ is the  graph:
$$
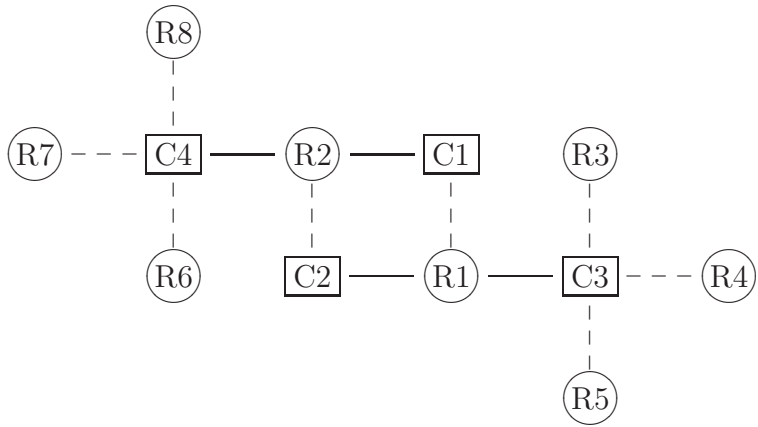
\xymatrix{
&
\begin{xy} *+{\txt{R8} }* \frm{o}
\end{xy}  \ar@{{}--{}}[d]
\\
\begin{xy} *+{\txt{R7} }* \frm{o}
\end{xy}  \ar@{{}--{}}[r] &
\fbox{C4}
\ar@{{}-{}}[r]
\ar@{{}--{}}[d] &
\begin{xy} *+{\txt{R2} }* \frm{o}
\end{xy}  \ar@{{}--{}}[d] \ar@{{}-{}}[r] & \fbox{C1}
\ar@{{}--{}}[d] &
\begin{xy} *+{\txt{R3} }* \frm{o}
\end{xy}  \ar@{{}--{}}[d]
\\
& \begin{xy} *+{\txt{R6} }* \frm{o}
\end{xy}
&
\fbox{C2} \ar@{{}-{}}[r] & \begin{xy}
*+{\txt{R1}}* \frm{o} \end{xy}\ar@{{}-{}}[r]
& \fbox{C3} \ar@{{}--{}}[r] \ar@{{}--{}}[d]
 & \begin{xy}
*+{\txt{R4}}* \frm{o} \end{xy}
\\
&&&&  \begin{xy}
*+{\txt{R5}}* \frm{o} \end{xy}
}
$$
There are 8 rows and 4 columns,
so
rank of $S$ and $S_\M$ is $4$. \
One e-cycle.

$G(S_\M)$ admits $3\cdot 3\cdot 2=18$
perfect matchings and the
e-cycle
is interlacing with respect to every one of those.
Each perfect matching selects a
$4\times 4$ submatrix of $S_\M$
(or $U_\M $) with one e-cycle in its graph.
In total there are 9 such submatrices
(each in $\cN$) with the
graph of each one admitting two perfect matchings.

Theorems \ref{thm:lowBound} plus \ref{thm:count} imply
$$ 
9 = \sum_{ U_i\in\cN } m (U_i)
  \leq 
m(S)    \leq  \sum_{ U_i\in\cN } [ t(U_i) - m(U_i) ]
=  9( 2 - 1) = 9.
$$
Thus generically $\cd(S)$ has $9$ anomalous signs.
\qed
\end{exam}

\subsection{Few Anomalous Signs - An Algorithm}
We have just looked at bounds for the number of anomalous
signs in $\cd(S)$ for generic $S_\M$.
A small number of anomalous signs in the core determinant can be
handled \emph{precisely} using an algorithm we now describe which
obtains
{\it
necessary and sufficient conditions for $\cd(S)$
to have $($zero or$)$ one anomalous sign.}

\subsubsection{\bf The zero-one anomalous sign algorithm:}
\label{1bsa}
Suppose $S$ is a $d\times d'$ matrix of rank $r$.
In order for the algorithm to work with certainty,
we assume $S_\M$ is generic.
Let $\cN$ be the set of all $r\times r$ submatrices of $S_\M$
that are not SD.
Given $S_i\in\cN$ we use $U_i$ to denote the corresponding
submatrix of $U_\M $.
We present the algorithm only for the case when $\cd(S)$ has
no anomalous signs
or the anomalous sign is $(-1)^{d-1}$.
\footnote{This assumption is made purely for convenience of exposition.
In fact, if $S_\M$ has at least two SNS $r\times r$ submatrices
with nonsingular corresponding submatrices in $U_\M $, then this will
automatically be the case.
}

\ben
\item[Case E:] {\it $\cN$ has $0$ elements.}\\
Then $\cd(S)$ has
no anomalous signs.
\item[Case N:] {\it $\cN$ is nonempty}.
\ben
\item[Subcase (a):] {\it All the $U_i$ corresponding to $S_i\in\cN$
are SD.}\\
Take $\det(S_i)\, \det(U_i)$ and look at its sign.
If for all $S_i\in\cN$ this sign is $(-1)^r$,
then $\cd(S)$ has no anomalous signs.
Otherwise for some $S_i\in\cN$ the sign is $(-1)^{r-1}$ and
the corresponding term $\det(S_i)\, \det(U_i)$ contributes
$t(U_i)$ terms with sign $(-1)^{r-1}$
to $\cd(S)$.
If $t(U_i)>1$, then there is more than one anomalous sign in $\cd(S)$.
If there is $S_j\ne S_i$ with $\sign (\det(S_j)\, \det(U_j))=(-1)^{r-1}$,
then $\cd(S)$ will have more than one anomalous sign.
Otherwise $\cd(S)$ has one anomalous sign.

\item[Subcase (b):] {\it There is exactly one $S_0\in\cN$ for which the corresponding $U_0$
is not SD.}\\
If there is $S_i\in\cN\setminus\{S_0\}$ with
the sign of $\det(S_i)\,\det(U_i)$ equal to $(-1)^r$, then $\cd(S)$ will have more than one
anomalous sign.
Otherwise we
use the det sign test (Theorem \ref{thm:count}) to compute $m(U_0)$.
\ben
\item[(i)]
If $m(U_0)>1$, then $\cd(S)$ will have more than one anomalous sign.
\item[(ii)]
Suppose $m(U_0)=1$.
If the number of terms $t$ in $\det(U_0)$ is two,
$\cd(S)$ will have one anomalous sign. So
suppose $t>2$.
Let
$$
\epsilon= \left\{ \begin{array}{lcl}
+1 &\mid& m(U_0)=m_+(U_0)\\
-1 &\mid& \text{otherwise}.
\end{array}\right.
$$
Now $\cd(S)$ will have one anomalous sign
iff
\begin{equation}\label{eq:wellBehaved}
\epsilon \sign \det(S_0) =(-1)^{r-1}.
\end{equation}
(If \eqref{eq:wellBehaved} fails, $\cd(S)$ will have more than one anomalous sign.)
\een

\item[Subcase (c):] {\it There are at least two $S_i\in\cN$ for which the corresponding $U_i$
is not SD.}\\
In this case $\cd(S)$ will have at least two anomalous signs.
\een
\een

\begin{lemma}
\label{lem:recipe1}
The {\it zero-one anomalous sign algorithm} computes whether or not
there is one $($respectively no$)$ anomalous sign.
\end{lemma}

\begin{proof}
Case E is given in Example \ref{ex:recipe0}.
Case N.(a) follows directly from the Binet-Cauchy formula \eqref{eq:rankdet2}.
For Case N.(b).(i), $\det(S_0)\, \det(U_0)$ has more than one anomalous sign,
so $\cd(S)$ will have more than one anomalous sign.
The proof of Case N.(b).(ii) is essentially contained in the statement.
Finally, in the Case N.(c) two different $S_i$ contribute
two different terms to the Binet-Cauchy expansion \ref{eq:rankdet2} for $\cd(S)$
each having at least one anomalous sign.
\end{proof}

\begin{rem}
The algorithm simplifies considerably in the fully reversible case, as then $U_i$
is SD iff $S_i$ is. Thus Case N.(a) cannot arise.
Subcase (b) is equivalent to $\cN$ having exactly one element and
Subcase (c) is equivalent to $\cN$ containing at least two elements.
\qed
\end{rem}

\begin{exam}
Let $S=\left[\begin{array}{cc}S_\M&-S_\M\end{array}\right]$, where
$$
S_\M=\left[\begin{array}{rrr}
a&1&0\\
1&1&1\\
-1&0&1\\
0&1&0
\end{array}\right]$$
has rank $3$
and $a\in\RR_{>0}$. Suppose $ a \neq 2$; this makes $S_\M$ generic. 
The graph $G(S_\M)$ is given by the following:
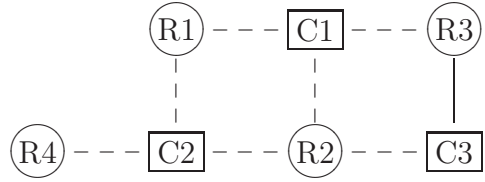
\[
\xymatrix{
&
\begin{xy} *+{\txt{R1} }* \frm{o}
\end{xy}  \ar@{{}--{}}[d] \ar@{{}--{}}[r] & \fbox{C1}
\ar@{{}--{}}[d]
&
\begin{xy} *+{ \txt{R3} }* \frm{o}
\end{xy} \ar@{{}--{}}[l]
\\ 
\begin{xy} *+{ \txt{R4} }* \frm{o}
\end{xy} \ar@{{}--{}}[r]
& \fbox{C2} \ar@{{}--{}}[r] & \begin{xy}
*+{\txt{R2}}* \frm{o} \end{xy}
&
\fbox{C3}
\ar@{{}--{}}[l] \ar@{{}-{}}[u]
}
 \]

\noindent The cycle R$1-$C$1-$R$2-$C$2-$R$1$ has two c-pairs and is an
e-cycle; it is the only e-cycle and it interlaces two perfect matchings.
Both leave out R$4$ and 
select the same $3\times 3$ submatrix $S_0$ of $S_\M$.
Hence $\cN=\{S_0\}$. 

To count the number of
anomalous signs in $\cd(S)$ apply the Algorithm \ref{1bsa}. Our
situation corresponds to Case N.(b) and we compute $m(U_0)$, where
$U_0$ is the $3\times 3$ submatrix of $U_\M$ corresponding to $S_0$.
By the det sign test, $m(U_0)=1$ and $t(U_0)=3$. It is
easy to see that $m(U_0)=m_-(U_0)$; thus by the zero-one anomalous sign
algorithm, $\cd(S)$ will have one anomalous sign iff $a-2=\det(S_0) <0$.
(If $a>2$, $\cd(S)$ will have two anomalous signs.)
\qed\end{exam}

Another class of examples is cycles with short hair, a notion
we now elucidate.
A subgraph $\Gamma$ of $G$ is said to have {\bf short hair}
provided when edges of $\Gamma$ are removed from $G$ all paths
in the remaining graph starting
from a vertex in $\Gamma$ have length $\leq 1$.

\begin{prop}\label{cor:1cycle}
Suppose $S_\M$ is generic and
the graph $G(S_\M)$ is connected and contains at most one cycle
and possibly some short hair $($e.g.~Example {\rm \ref{ex:CF3}} or Example {\rm \cite[Table 1.1.(iii)]{CF05})}.
Then the number of anomalous signs in $\cd(S)$ is $\leq 1$.
\end{prop}

\begin{proof}
Without loss of generality, $G(S_\M)$ contains a cycle $\cE$.
If $\cE$ is not an e-cycle, then
we are in Case E of the algorithm and there are no
anomalous signs in $\cd(S)$. Thus we assume that
$\cE$ is an e-cycle.

By Lemma \ref{lem:genericRank}, the rank $r$ of $S_\M$ is the minimal number
of rows or of columns in $S_\M$.
If $r$ is bigger than the number of columns appearing in
the cycle, then
we are in Case E of the algorithm because
any perfect matching will include some edge not in the cycle,
thus
making $\cE$ not interlace it. Hence $\cd(S)$ has no anomalous signs.

Otherwise $r$ equals the number of columns appearing in
the cycle.
Then $\cN$ has only one element $S_0$.
The corresponding submatrix $U_0$ of $U_\M $ is either
SD or its determinant expansion has two terms of opposite sign.
Now the result follows
from Case N.(a).
\end{proof}

Note that results on the core determinant $\cd(S)$ given
in this section have parallels for the
Craciun-Feinberg determinant expansion $\cfd(S)$ which
are easy to work out using the techniques in our paper.

\section{Chemical Motivation}
\label{sec:ChemMotiv}
\label{subsec:coreCF}

This matrix theory paper is not directly aimed
at producing chemical results but was inspired as an extension
of the striking work of Craciun and Feinberg.
We hope these extensions might someday prove valuable
on chemical network problems and some methods they combine with
are described in \cite{CHWprept} and a consequence
is Theorem \ref{thm:degr} below.

Now we turn to describing the connection between
the core determinant from
\S \ref{sec:Jac} and chemistry.

A chemical reactor can be thought of as a tank
with each chemical species  flowing in (assume at a constant rate)
and each species flowing out (assume in proportion to
its concentration in the tank).
If the reaction inside the tank satisfies
$  \frac{d x}{dt}= g(x)$,
then when there are inflows and outflows, the
total reaction satisfies
$$ \frac{d x}{dt}= f(x)= g(x) + \eps x_{\rm in}  - \delta x.$$
The Craciun-Feinberg determinant is the determinant  of the Jacobian
$f'$ when $\delta$ is 1 and it bears on counting the number of equilibria
 for this differential equation,
 cf.~\cite{CF05,CF06, CHWprept}.
There is some discussion of small outflows vs.~no
outflows in \cite{CF06iee}.

The core determinant bears on a different problem.
Assume the differential equation has reaction form $f(x)=S v(x)$.
Let $R$ (respectively $R^\perp$)
denote the range of $S$ (respectively its orthogonal complement);
$R$ is typically called the stoichiometric subspace.
Let  $P$ be the projection onto $R$ and $P^\perp$ onto $R^\perp$.
With no inflows and outflows, $P^\perp f(x)=0$ and clearly this
implies
the solution $x(t)$ to the differential equation
propagates on the affine subspace
\beq
\cM_{x^0} := \{ x \mid  P^\perp x(t) = \text{const} = P^\perp x^0 \}
.
\eeq
This reflects  quantities (like the number of carbon atoms)
being  conserved.
The flow on $\cM_{x^0}$ has dynamics
$\frac{ d Px}{dt} = \frac {d(Px + P^\perp x^0)}{dt} = P f(Px + P^\perp x^0)$.
Proposition \ref{prop:BinetCauhr-ranks}
implies that the determinant of the Jacobian of this dynamics
is the core determinant
which we studied in this paper, namely, for any $\xi$ in $\cM_{x^0}$
\beq
\cd(S)(\xi) = \det (P f'(\xi)P).
\eeq

When $\cd(S)$ has no anomalous signs
the degree theory arguments in \S 3 of \cite{CHWprept}
give a strong result for numbers of equilibria of
the differential equation.

 \begin{theorem}
 \label{thm:degr}
Suppose $\frac{dx}{dt}= f_b(x): = Sv^b(x)$ has reaction form
with $v^b(x)$ once continuously differentiable in $x$
and depending
continuously on a parameter $0 \leq b\leq 1$.
Suppose each component $v^b_j(x)$ of $v^b(x)$
 is monotone nondecreasing.
Suppose $\cM_{x^0}$ is compact.
Suppose $\cd(S)$ has no anomalous signs.

If there are  no zeroes $f_b(x)=0$ for any $b$
 and any $x$ on the boundary of $\cM_{x^0}$,
then the number of zeroes for $f_b$ in the interior of
$\cM_{x^0}$
is independent of $b$.
\end{theorem}

The hypothesis that $\cd(S)$ has no anomalous signs
can be weakened to $\cd(S)(\xi)$ does not equal 0
for  any $\xi$ in $\cM_{x^0}$.

\newpage

\centerline{NOT FOR PUBLICATION}
\today

\tableofcontents

\end{document}